\documentclass[12pt]{article}

\usepackage{amsmath,amssymb,amsthm,mathtools,mathrsfs,latexsym}
\usepackage{graphicx}
\usepackage{xcolor}
\usepackage{tikz}
\usepackage[linkcolor=blue,colorlinks=true,urlcolor=red,bookmarksopen=true]{hyperref}
\usepackage{newclude}
\usepackage{imakeidx}
\makeindex

\allowdisplaybreaks

 \numberwithin{equation}{section}

\newtheorem{thm}{Theorem}[section]

\newtheorem{lema}{Lemma}[section]
\newtheorem{prop}{Proposition}[section]

\newtheorem{rmk}{Remark}[section]

\renewcommand{\Pi}{\mathcal{P}}

\renewcommand{\Xi}{\mathcal{X}}

\newcommand{\R}{\mathbb{R}}
\renewcommand{\S}{\mathbb{S}}
\newcommand{\I}{\mathbb{I}}
\newcommand{\D}{\mathbb{D}}

\newcommand{\T}{\mathbb{T}}

\newcommand{\ub}{\bold{u}}
\newcommand{\xb}{\bold{x}}
\newcommand{\nb}{\bold{n}}

\let\div\relax
\DeclareMathOperator*{\div}{div}
\DeclareMathOperator*{\curl}{curl}

\DeclarePairedDelimiter\abs{\lvert}{\rvert}
\DeclarePairedDelimiter\norm{\lVert}{\rVert}
\makeatletter
\let\oldabs\abs
\def\abs{\@ifstar{\oldabs}{\oldabs*}}
\let\oldnorm\norm
\def\norm{\@ifstar{\oldnorm}{\oldnorm*}}
\makeatother

\usepackage{scalerel}
\usepackage[usestackEOL]{stackengine}
\def\dashint{\,\ThisStyle{\ensurestackMath{
\stackinset{c}{.2\LMpt}{c}{.5\LMpt}{\SavedStyle-}{\SavedStyle\phantom{\int}}
}
\setbox0=\hbox{$\SavedStyle\int\,$}\kern-\wd0}\int}

\newcommand{\br}[1]{\left(#1\right)}

\newcommand{\set}[1]{\left\{#1\right\}}

\usepackage{adjustbox}
\usepackage{wrapfig}
\usepackage{enumitem}
\usepackage{protosem}

\title{Free boundary value problem for the radial symmetric\\ compressible isentropic Navier-Stokes equations\\
with density-dependent viscosity}
\date{}
\author{
\bf\large Xiangdi Huang$^a$, Weili Meng$^a$, Anchun Ni$^a$*\\
\small a. Institute of Mathematics, Academy of Mathematics and Systems Sciences\\
\small Chinese Academy of Sciences, Beijing 100080, P.R.China;\\
\footnote{*Email addresses: xdhuang@amss.ac.cn\,\, (X. Huang); mengweili@amss.ac.cn\,\, (W. Meng); varnothing@outlook.com\, \, (A. Ni); } }
\begin{document}
\maketitle

\begin{abstract}
This paper is devoted to the study of free-boundary-value problem of the compressible Naiver-Stokes system with density-dependent viscosities $\mu=const>0,\lambda=\rho^\beta$ which was first introduced by Vaigant-Kazhikhov \cite{1995 Vaigant-Kazhikhov-SMJ} in 1995. By assuming the endpoint case $\beta=1$ in the radially spherical symmetric setting, we prove the (a priori) expanding rate of the free boundary is algebraic for multi-dimensional flow, and particularly establish the global existence of strong solution of the two-dimensional system for any large initial data. This also improves the previous work of Li-Zhang \cite{2016 Li-Zhang-JDE} where they proved the similar result for $\beta>1$. The main ingredients of this article is making full use of the geometric advantage of domain as well as the critical space dimension two. \\[4mm]
\end{abstract}

\section{Introduction}
We study the free-boundary-value problem of the following multi-dimensional compressible viscous fluids:
\begin{equation}
\label{CNS}
	\begin{dcases}
		\rho_t +\div (\rho \ub )=0, & \text{ in }\Omega_t,\\
		(\rho \ub)_t +\div(\rho \ub \otimes \ub)=\div\S, & \text{ in }\Omega_t,
	\end{dcases}
\end{equation}
with rheological assumption and the stress-free boundary condition:
\begin{equation}
	\S=-P(\rho) \I +2\mu \D \ub +\lambda (\rho) \div \ub \I,
\end{equation}
\begin{equation}
	P(\rho)=\rho^\gamma, \mu=const, \lambda(\rho)=\rho, \gamma>1, 
\end{equation}
\begin{equation}
\label{stress free}
	\S \cdot \nb=0,  \text{ on }\Gamma_t,
\end{equation}
where $\Gamma_t $ is the free surface of domain $\Omega_t\subset \R^d(d=2,3)$. Here $\rho (\xb,t), \ub(\xb,t), P(\xb,t)$ stand for the fluid density, velocity and pressure respectively, $\S=\S(\rho,\ub)$ is the total stress tensor and $\mu,\lambda$ are the bulk and shear viscosity coefficients which may depend on the density. 

The compressible Navier-Stokes equations with density-dependent viscosity $\mu=\mu(\rho), \lambda=\lambda(\rho)$ have attracted a lot of attention due to the physical importance and rich phenomena, which can be derived from the fluid-dynamical approximation to the Boltzmann equation. A special case  $\mu(\rho)=\rho,\lambda(\rho)=0,\gamma=2$ and $d=2$ expresses  exactly the viscous Saint-Venant system describing the motion for shallow water where $\rho,\ub$ represent the height and the average velocity, see \cite{2001 Gerbeau-Perthame-DCDSSB, 2007 Marche-EJMB} for details. The system considered in this paper follows the model proposed by Vaigant-Kazhikhov \cite{1995 Vaigant-Kazhikhov-SMJ} assuming that  $\lambda(\rho)=\rho^\beta$ and $\mu$ is a positive constant.

The research on the well-posedness and long-time dynamic behavior of compressible viscous fluids with constant viscosity has significant progress in recent years: see \cite{2024 Huang-Li-Zhang-Preprint1,  2012 Huang-Li-Xin-CPAM, 1980 Matsumura-Nishida-JMKU, 2022 MRRS-AM-1, 2022 MRRS-AM-2, 2008 Razanova-JDE,  1998 Xin-CPAM, 2013 Xin-Yan-CMP, 2023 Yu-MMAS}. We briefly recall some results about the strong solution in multi-dimensional cases: Matsumura-Nishida \cite{1980 Matsumura-Nishida-JMKU} establish the global well-posedness of classical solution in $3$D when the initial regular data is close to a equilibrium state, and the case when the initial data with small energy but possible large oscillations and vacuum is proved by Huang-Li-Xin \cite{2012 Huang-Li-Xin-CPAM}, the case when density is large is proved by Yu \cite{2023 Yu-MMAS} and Huang-Li-Zhang \cite{2024 Huang-Li-Zhang-Preprint1}. When the sufficiently small or large assumptions fail, Xin \cite{1998 Xin-CPAM} and later \cite{2008 Razanova-JDE,2013 Xin-Yan-CMP} prove the finite-time blowup from strongly vanishing smooth data. Very recently, Merle-Rapha\"el-Rodnianski-Szeftel \cite{2022 MRRS-AM-1, 2022 MRRS-AM-2} construct smooth implosions without vacuum for both Euler  and Navier-Stokes equations in 2D and 3D cases, see \cite{2021 Biasi-CNSNS, 2023 CGSS-preprint} for the later improvements.

As for  the density-dependent viscosity flow $\mu=\mu (\rho),\lambda=\lambda(\rho)$, the investigation becomes complicate due to the high nonlinearity and possible degeneracy. Vassuer-Yu \cite{2016 Vasseur-Yu-IM}, and  Li-Xin \cite{2015 Li-Xin-Preprint} independently prove the global existence of the weak solution for the density-degenerate case, based on the significant estimates by Bresch-Desjardins \cite{2003 Bresch-Desjardins-CMP} and Mellet-Vassuer \cite{2007 Mellet-Vasseur-CPDE}, when the global well-posedness of classical solutions are still largely open. 

Compared to the density-degenerate case, the regularity of the solution of  system with $\lambda=\rho^\beta,\mu>0$ is more expectable because of the uniform parabolicity. See \cite{2016 Huang-Li-JMPA, 1995 Vaigant-Kazhikhov-SMJ} for the global existence of classical solution in multidimensional cases and under the fixed boundary assumption, and \cite{2024 Huang-Li-Zhang-Preprint1} for the case $\beta>1$ in the extra spherically symmetric setting and under the Dirichlet boundary condition.

When it comes to the free-boundary-value problem in our setting, many new challenges appear because of the influence of the vacuum state. The global regularity  of three-dimensional system with small amplitude and non-vacuum in fluid region is established,  under the stress-balanced boundary condition by surface tension and/or exterior pressure \cite{1992 Solonnikov-Tani, 1993 Zajaczkowski-DM, 1994 Zajaczkowski-JMA}. Without the smallness assumption, the global existence is expectable in one-dimensional or spherically symmetric cases. See \cite{2008 Guo-Jiu-Xin-JMA, 2005 Jiang-Xin-Zhang-MAA} for more details.

In our paper, we consider the system in radially spherically symmetric setting:
\begin{equation}
	\rho(\xb,t)=\rho (r,t),\quad \ub(\xb,t)=u(r,t)\frac{\xb}{r},\quad r=|\xb|.
\end{equation}
And then the stress-free boundary value problem \eqref{CNS}-\eqref{stress free} is changed to
\begin{equation}
\label{radial CNS}
	\begin{dcases}
		\rho_t +(\rho u)_r +\frac {d-1}r(\rho u)=0,\\
		(\rho u)_t + (\rho u^2)_r +\frac {d-1}r \rho u^2 + (\rho^\gamma)_r= \br{(2\mu +\rho)(u_r +\frac {d-1}r u) }_r,
	\end{dcases}
\end{equation}
with the initial data
\begin{align}
    \rho(r,0)=\rho_0(r),\quad u(r,0)=u_0(r), \quad 0\leq r<a_0,
\end{align}
and the boundary condition:
\begin{equation}
\label{stress free condition}
	u(0,t)=0,\quad \br{\rho^\gamma-(2\mu+\rho)(u_r +\frac {d-1}r u)}(a(t),t)=0,
\end{equation}
where $a(t)$ satisfies
\begin{equation}
\label{at}
	a'(t)=u(a(t),t),
\end{equation}
 corresponding to the free surface of $\Omega_t$, that is:
\begin{equation*}
	\Omega_t =\set{(r,t)| 0\le r\le a(t), 0\le t\le T}.
\end{equation*}

The well-posedness and Lagrangian structure to the spherically symmetric compressible flow with density-dependent viscosity and stress-free boundary  has been studied in recent years. For the density-degenerate viscosity $\mu=\rho, \lambda=0$, Guo-Li-Xin \cite{2012 Guo-Li-Xin-CMP} prove the global existence of the entropy weak solution, and they prove also the non-formation of vacuum state in finite time and that the free surface expands at an algebraic rate. As for the case when $\mu$ is a positive constant and $\lambda=\rho^\beta,\beta>1$, Li-Zhang \cite{2016 Li-Zhang-JDE} establish the global well-posedness of strong solution and the algebraic expanding rate.

Our goal is to establish the well-posedness result and analyze the dynamic behavior in the endpoint case $\beta=1$, where the first result can be stated as following:

\begin{thm}
\label{expanding rate}
Suppose $d=2$ or $3$, and $(\rho, \ub ,a)$ is a global strong solution of the free-boundary-value problem \eqref{radial CNS}-\eqref{at}, then the expanding rate of $a(t)$ is algebraic. Indeed, we have
\begin{equation}
\label{a bound}
	\begin{split}
		a(t)\ge \begin{dcases}
			C(1+t)^{\frac{1}{d\gamma}}, & \gamma>1+\frac 1d,\\
			C\br{\frac{1+t}{1+\log (1+t)} }^{\frac 1{d\gamma}}, & \gamma=1+\frac 1d.
		\end{dcases}
	\end{split}
\end{equation}
Moreover, denote $a_M(t)=\sup_{s\in [0,t]}a(s)$, then
\begin{equation}
\label{aM bound}
	a_M(t)\ge \begin{dcases}
		C(1+t)^{\frac{1}{d\gamma}}, & \gamma>1+\frac 1d,\\
		C\br{\frac{1+t}{1+\log (1+t)} }^{\frac 1{d\gamma}}, & \gamma=1+\frac 1d,\\
		C(1+t)^{\frac{\gamma-1}{\gamma}},  & \gamma\in (1,1+\frac 1d).
	\end{dcases}
\end{equation}
\end{thm}
\begin{rmk}
    In \cite{2016 Li-Zhang-JDE} , Li-Zhang proved an analogous result of Theorem \ref{expanding rate} in 2D case. Following their ideas, the algebraic expanding rate can be also established for 3D flows.
\end{rmk}
Now we are ready to state our main theorem as follows:
\begin{thm}
\label{global existence}
Suppose that $\gamma>1,d=2$. Assume that the initial data $(\rho_0,\ub_0)$ satisfy
\begin{equation}
	\label{initial data}
	\rho_0\in L^1(\Omega_0)\cap W^{1,q}(\Omega_0),\quad \ub_0\in H^2(\Omega_0),\quad q>2.
\end{equation}
and
\begin{align}
    \rho_0(r)>0 \quad\text{for} \quad r\in[0,a_0), \quad\rho_0(a_0)\geq 0, \quad \text{and} \int_{\Omega_0}\rho_0 dx=1.
\end{align}
Then for any $T>0$, the free-boundary-value problem \eqref{radial CNS}-\eqref{at} has a unique global spherically symmetric strong solution  $(\rho,\ub)$ satisfying
\begin{equation*}
	\begin{split}
		\rho \in L^\infty (0,T; L^1(\Omega_t)\cap W^{1,q}(\Omega_t)), \ub\in L^\infty (0,T; H^2(\Omega_t)).
	\end{split}
\end{equation*}
Moreover, it holds that
\begin{equation*}
	\begin{split}
				&\sup_{t\in [0,T]}\int_{\Omega_t}\br{\frac 12 \rho |\ub |^2 +\frac1{\gamma-1}\rho^\gamma  }d\xb+\int_0^T \int_{\Omega_t} (2\mu +\rho)|\div \ub |^2 d\xb dt\le C,\\
                &\sup_{t\in [0,T]}\int_{\Omega_t}F^2 d\xb +\int_0^T \int_{\Omega_t} \rho |\dot \ub |^2 d\xb dt\le C,\\
                &\sup_{t\in [0,T]}\int_{\Omega_t}\rho |\dot \ub |^2 d\xb +\int_0^T \int_{\Omega_t} (2\mu+\rho)|\div \dot \ub |^2 d\xb dt\le C,\\
                &\sup_{t\in [0,T]}\br{ \norm{\rho}_{L^1(\Omega_t)\cap W^{1,p}(\Omega_t)} +\norm{\ub}_{H^2(\Omega_t)} }+\int_0^T \br{\norm{\nabla \ub }_{L^\infty(\Omega_t)}^2+\norm{\ub_t}_{H^1(\Omega_t)}^2 }dt\le C.
	\end{split}
\end{equation*}
where $F$ denotes the effective viscous flux, which is given by
$$F\eqqcolon (2\mu+\rho)\div \ub-P(\rho). $$
\end{thm}
Now we comment on the analysis of this paper.

\begin{rmk}
    In Kazhikhov's original paper \cite{1995 Vaigant-Kazhikhov-SMJ}, the $L^p$ bound of the density is only established under the condition $\beta>1$. This was also proved by Li-Zhang \cite{2016 Li-Zhang-JDE} for the free boundary value problem under the same assumption $\beta>1$. Thus left an open problem whether the same result hold for the endpoint case $\beta=1$, this is the main problem we shall address in this paper. Indeed, for the endpoint case $\beta=1$, one of the main obstacles is to prove higher integrability of the density and velocities. This was achieved by using the special structure of radially symmetry, such as $\norm{\ub}_{L^\infty}\lesssim \norm{\nabla \ub}_{L^2}$. see Lemma \ref{lem 4.1} and \ref{lem 4.2}. Then the next crucial part is dedicated to the lower and upper bound of the density, i.e, to establish Gronwall's inequality for the density. Compared to Li-Zhang \cite{2016 Li-Zhang-JDE}, we use $\norm{\rho}_{L^\infty}$ instead of $\norm{\rho}_{L^\infty L^\infty}$ to bound $\int_{a(t)}^r\frac{\rho u^2}{s}ds$.
\end{rmk}

The rest of the paper is organized as follows. 

In Section \ref{sec 2}, we give some important observations and basic tools as preliminary; In Section \ref{sec 3}, we prove the algebraic expanding rate of the free domain in Theorem \ref{expanding rate}; In Section \ref{sec 4}, we establish the upper bound of the density in two-dimensional case, which is crucial to obtain the global regularity. In Section \ref{sec 5}, we obtain the high order estimates and then finish the proof of Theorem \ref{global existence}.

\section{Preliminary}
\label{sec 2}
In this section, we will recall some known facts and important lemmas that will be used to prove the main results.
\begin{lema}
	Assume that $\ub(\xb, t)$ satisfies the radial setting and the boundary condition $\ub(0, t)=0$, then it holds that
		\begin{equation}
		 \norm{\ub(t)}_{L^\infty(\Omega_t)}=\norm{u(t)}_{L^\infty(0,a(t))}\le \norm{\nabla \ub(t) }_{L^2(\Omega_t)}\le \norm{\div \ub (t)}_{L^2(\Omega_t)} .
	\end{equation}
\end{lema}
\begin{proof}
	We observe for any $r\in [0,a(t)]$,
	\begin{equation*}
		\begin{split}
			|u(r,t)|^2=& 2\int_0^r u \partial_r u (s,t)ds\le 2\int_0^{a(t)} \abs{\partial_r u(s,t) }\abs{\frac{u(s,t)}{s} }sds\\
			\le & \int_0^{a(t)} \br{\abs{\partial_r u }^2 s +\abs{\frac us }^2 s } ds\le \int_0^{a(t)}\br{\partial_r u+\frac{u}{s} }^2 sds\\
			=&\int_{\Omega_t}\abs{\div \ub(\xb,t) }^2 d\xb = \norm{\div \ub(t) }_{L^2(\Omega_t)}^2.
		\end{split}
	\end{equation*}
	Here the fourth step we cancel the cross term using the following fact:
	\begin{equation*}
		\begin{split}
			\int_0^{a(t)} 2u\partial_r u ds=\int_0^{a(t)}\partial_r (u^2)ds= u^2(a(t),t)\ge 0.
		\end{split}
	\end{equation*}
	Combining together, we obtain the above.
\end{proof}

\begin{lema}
	Let $(\rho, \ub, a(t))$ be a solution of free boundary system \eqref{radial CNS}-\eqref{at}, then  it satisfies the following
	mass conservation:
	\begin{equation}
	\label{mass conservation}
		\int_0^{a(t)}\rho r^{d-1}dr\equiv \int_0^{a_0}\rho_0 r^{d-1}dr=1,
	\end{equation}
 and the energy inequality:
	\begin{equation}
	\label{energy inequality}
		\sup_{t\in [0,T]}\int_{\Omega_t}\br{\frac 12 \rho |\ub |^2 +\frac1{\gamma-1}\rho^\gamma  }d\xb+\int_0^T \int_{\Omega_t} (2\mu +\rho)|\div \ub |^2 d\xb dt\le E_0,
	\end{equation}
	where $E_0$ depends only on initial datas $(\rho_0, u_0,a_0)$.
\end{lema}

\begin{lema}(\cite{1984 Beal-Kato-Majda-CMP})\label{lem 2.3}
For any $2<q<\infty$, there is a constant $C(q)$ such that for all $\nabla\mathbf{u}\in W^{1,q}(\Omega)$, there holds 
\begin{align*}
\norm{\nabla\mathbf{u}}_{L^\infty(\Omega)}\leq C(\norm{\div \mathbf{u}}_{L^\infty(\Omega)}+\norm{\curl \mathbf{u}}_{L^\infty(\Omega)})\log(e+\norm{\nabla^2\mathbf{u}}_{L^q(\Omega)})+C\norm{\nabla\mathbf{u}}_{L^2(\Omega)}+C
\end{align*}
\end{lema}

\begin{lema}(\cite{2004 Feireisl})\label{lem 2.4}
    Let $\mathbf{u}\in W^{1,2}(\Omega)$, and let $\rho$ be a non-negative function such that 
    \begin{align*}
        0<M_1\leq \int_{\Omega}\rho dx,\quad \int_{\Omega}\rho^\gamma dx\leq M_2
    \end{align*}
    where $\Omega\subset \mathbb{R}^2$ is a bounded domain and $\gamma>1$, $M_1>0$, $M_2>0$. Then there exist a constant $C$ depending only on $M_1,M_2$ and $\Omega$ such that
    \begin{align*}
        \norm{\mathbf{u}}_{L^2(\Omega)}^2\leq C\left(\norm{\nabla\mathbf{u}}_{L^2(\Omega)}^2+\left(\int_\Omega\rho|\mathbf{u}|dx\right)^2\right).
    \end{align*}
\end{lema}

\section{Proof of Theorem \ref{expanding rate}}
\label{sec 3}

This section is devoted to the proof of Theorem \ref{expanding rate}. \\
The upper bound of $a(t)$ is directly from the energy inequality \eqref{energy inequality} in the case $d=2$:
\begin{equation}
\label{at lower bound}
	\begin{split}
		a(t)=&a_0+\int_0^t a'(s)ds\le a_0+Ct^\frac 12\br{\int_0^t |a'(s)|^2ds }^\frac 12\\
		\le & C\br{1+t^\frac12\int_0^t u^2(a(s),s)ds }\le C(1+t)^\frac 12,
	\end{split}
\end{equation}
where in the last step
\begin{equation*}
	\begin{split}
		\int_0^t u^2(a(s),s)ds=&\int_0^t \int_0^{a(s)}2 u u_r dr ds\le \int_0^t \int_0^{a(s)}(u_r+\frac ur)^2 rdrds\\
		\le &\int_0^t \int_{\Omega_s}|\div \ub |^2 d\xb ds\le E_0.
	\end{split}
\end{equation*}
Now we attempt to derive the lower bound for $a(t)$ by setting the potential energy as the mediate term. On the one hand, we can deduce from the mass conservation \eqref{mass conservation}:
\begin{equation*}
	\begin{split}
		1\equiv & \int_0^{a(t)}\rho r^{d-1}dr= \int_0^{a(t)}r^{(d-1)\frac{\gamma-1}{\gamma}}\cdot \rho r^{\frac{ d-1}{\gamma}}dr\\
		\le &C(a(t))^{\frac{d(\gamma-1)}{\gamma}}\br{\int_0^{a(t)}\rho^\gamma r^{d-1}dr}^{\frac{1}{\gamma}},
	\end{split}
\end{equation*}
which yields that 
\begin{equation}\label{a lower bound}
    (a(t))^{-d(\gamma-1)}\leq C\int_0^{a(t)}\rho^\gamma r^{d-1}dr.
\end{equation}
Therefore, there exists a constant $C>0$ depends only on $E_0,\gamma,$ and $d$ such that
\begin{align}
    a(t)\geq C.
\end{align}

On the other hand, to obtain the lower bound of $a(t)$, the key idea is to define a energy functional for the spherical symmetric solution as
\begin{equation}
\label{energy functional}
	\begin{split}
		H(t)=&\int_0^{a(t)}(r-(1+t)u)^2\rho r^{d-1}dr+2(1+t)^2\int_0^{a(t)}\frac{\rho^\gamma}{\gamma-1}r^{d-1}dr\\
		=& \int_0^{a(t)} \rho r^{d+1}dr-2(1+t)\int_0^{a(t)}\rho u r^d dr\\
	&+(1+t)^2\int_0^{a(t)}\rho u^2r^{d-1}dr+2(1+t)^2\int_0^{a(t)}\frac{\rho^\gamma}{\gamma-1}r^{d-1}dr.
	\end{split}
\end{equation}
We can check directly that
\begin{equation}
\label{functional derivative}
	\begin{split}
		H'(t)&=\int_0^{a(t)}(\rho_t r^2-2\rho u r) r^{d-1}dr+2(1+t)\int_0^{a(t)}\br{\rho u^2+\frac{2\rho^\gamma}{\gamma-1}-(\rho u)_t r }r^{d-1}dr\\
		&\quad+(1+t)^2 \int_0^{a(t)}\br{\rho u^2 +\frac{2\rho^\gamma}{\gamma-1} }_t r^{d-1}dr\\
        &\quad+\rho(a(t))^{d+1}u(a(t),t)-2(1+t)\rho(a(t))^du^2(a(t),t)+(1+t)^2\rho (a(t))^{d-1}u^3(a(t),t)\\
        &\quad+\frac{2}{\gamma-1}(1+t)^2\rho^\gamma(a(t))^{d-1}u(a(t),t)\\
        &\coloneqq I_0(t) +I_1(t)+I_2(t)+I_B(t).
	\end{split}
\end{equation}
Due to the stress-free boundary condition \eqref{stress free condition}, one obtains
\begin{equation*}
	\begin{split}
		I_0 (t)=&-\int_0^{a(t)} \br{(\rho u)_r r^2+(d-1)\rho ur+2\rho u r}r^{d-1}dr\\
		=&-\int_0^{a(t)}\br{\rho u r^{d+1} }_r dr=-\rho(a(t))^{d+1}u(a(t),t) ,\\
		I_1 (t)
        =&2(1+t)\int_0^{a(t)}\br{\rho u^2+\frac{2\rho^\gamma}{\gamma-1}+\partial_r(\rho u^2)r+(d-1)\rho u^2-\partial_rFr}r^{d-1} dr\\
        =&2(1+t)\int_0^{a(t)}\br{\rho u^2+\frac{2\rho^\gamma}{\gamma-1}-d\rho u^2+(d-1)\rho u^2+dF}r^{d-1} dr\\
        &+2(1+t)\rho(a(t))^du^2(a(t),t)\\
        =&2(1+t)\int_0^{a(t)}\br{\br{\frac{2}{\gamma-1}-d} \rho^\gamma r^{d-1}+d(2\mu+\rho)(r^{d-1}u)_r }dr \\
        &+2(1+t)\rho(a(t))^du^2(a(t),t),\\
    	I_2(t)
        =&(1+t)^2 \int_0^{a(t)}\br{2u\partial_t(\rho u)-u^2\partial_t\rho+\frac{2}{\gamma-1}\partial_t(\rho^\gamma)}r^{d-1}dr\\
        =&(1+t)^2 \int_0^{a(t)} 2u(-\partial_r(\rho u^2)-\frac{d-1}{r}\rho u^2+\partial_rF)r^{d-1}dr\\
        &+(1+t)^2 \int_0^{a(t)}\left(u^2-\frac{2\gamma}{\gamma-1}\rho^{\gamma-1}\right)\left((\rho u)_r+\frac{d-1}{r}\rho u\right)r^{d-1}dr\\
        =&-2(1+t)^2 \int_0^{a(t)} (2\mu+\rho)\br{r^{d-1} u }^2_r \frac 1{ r^{d-1}}dr-(1+t)^2\rho (a(t))^{d-1}u^3(a(t),t)\\
        &-\frac{2}{\gamma-1}(1+t)^2\rho^\gamma(a(t))^{d-1}u(a(t),t).
	\end{split}
\end{equation*}
It follows after substituting the above computations into \eqref{functional derivative}:
\begin{equation}
\label{H derivative}
	\begin{split}
		H'(t)=&{2(\frac {2}{\gamma-1}-d)}(1+t)\int_0^{a(t)}\rho^\gamma r^{d-1}dr+2d(1+t)\int_0^{a(t)}(2\mu+\rho)(r^{d-1}u)_r dr\\
        &-2(1+t)^2 \int_0^{a(t)} (2\mu+\rho)\br{r^{d-1} u }^2_r \frac 1{ r^{d-1}}dr
	\end{split}
\end{equation}
then
\begin{equation}
\label{H derivative'}
	\begin{split}
		H'(t)\leq& {2(\frac {2}{\gamma-1}-d)}(1+t)\int_0^{a(t)}\rho^\gamma r^{d-1}dr+4d\mu (1+t) \int_0^{a(t)}(r^{d-1}u)_rdr\\
		& +2d(1+t)\int_0^{a(t)}\rho(r^{d-1}u)_r dr-2(1+t)^2 \int_0^{a(t)} \rho\br{r^{d-1} u }^2_r \frac 1{ r^{d-1}}dr\\
		\le & {2(\frac {2}{\gamma-1}-d)}(1+t)\int_0^{a(t)}\rho^\gamma r^{d-1}dr+4d\mu (1+t) r^{d-1}u(a(t),t)+\frac{d^2}{2}\int_0^{a(t)}\rho r^{d-1}dr\\
		\le & {2(\frac {2}{\gamma-1}-d)}(1+t)\int_0^{a(t)}\rho^\gamma r^{d-1}dr +4\mu (1+t) \frac{d}{dt}\br{a^d(t)} +\frac{d^2}{2}M_0.
	\end{split}
\end{equation}
Here we used \eqref{at}, \eqref{mass conservation} and the decomposition
\begin{equation}
\label{decomposition}
	\begin{split}
		2d(1+t)(r^{d-1}u)_r\le 2(1+t)^2(r^{d-1}u)_r^2 \frac{1}{r^{d-1}}+\frac{d^2}{2}r^{d-1}.
	\end{split}
\end{equation}
For $\frac {2}{\gamma-1}-d\le 0$, or equivalently  $\gamma\ge 1+\frac 2d$, it yields that
\begin{equation*}
	\begin{split}
		H'(t)\le & C(1+t) \frac{d}{dt}\br{a^d(t)} +C\\
		\le & C\frac{d}{dt}\br{(1+t)a^d(t)}-C a^{d}(t)+C\\
		\le & C\frac{d}{dt}\br{(1+t)a^d(t)}+C.
	\end{split}
\end{equation*}
Then we integrate above with respect to $t$ and obtain
\begin{equation*}
	\begin{split}
		H (t)\le C(1+t)a^d(t)+Ct+C\le C(1+t)a^d(t).
	\end{split}
\end{equation*}
Together with \eqref{a lower bound} and \eqref{energy functional}, it leads to
\begin{equation*}
	\begin{split}
		\br{a(t)}^{-d(\gamma-1)}\le C\int_0^{a(t)}	\rho^{\gamma} r^{d-1}dr\le C(1+t)^{-1}a^d(t),
	\end{split}
\end{equation*}
which derive the lower bound \eqref{a bound}$_1$. And in the case $\gamma\in [1+\frac 1d,1+\frac 2d)$, which implies
$$d(\gamma-1)-1\ge 0,$$
we can deduce from \eqref{H derivative'} that
\begin{equation*}
	\begin{split}
		((1+t)^{-(2-d(\gamma-1))}H(t))'\le & C(1+t)^{-(2-d(\gamma-1))}((1+t)\frac{d}{dt}\br{a^d(t)}+\frac {d^2}2 )\\
		\le & C\frac{d}{dt}\br{(1+t)^{d(\gamma-1)-1}a^d(t)}+C(1+t)^{-(2-d(\gamma-1))}\\
		&-C\br{d(\gamma-1)-1}(1+t)^{-(2-d(\gamma-1))}a^d(t)\\
		\le & C\frac{d}{dt}\br{(1+t)^{d(\gamma-1)-1}a^d(t)}+C(1+t)^{-(2-d(\gamma-1))}.
	\end{split}
\end{equation*}
Then it comes for $\gamma\in (1+\frac 1d,1+\frac 2d):$
\begin{equation}
\label{gamma2}
	\begin{split}
		\int_0^{a(t)}\rho^\gamma r^{d-1}dr\le  Ca^d(t)(1+t)^{-1}+C(1+t)^{-1}\le C(1+t)^{-1}a^d(t), \\
	\end{split}
\end{equation}
and for $\gamma=1+\frac 1d:$
\begin{equation}
\label{gamma3}
\begin{split}
	\int_0^{a(t)}\rho^\gamma r^{d-1}dr\le & Ca^d(t)(1+t)^{-1}+C(1+t)^{-1}\log(1+t)\\
	\le & C(1+t)^{-1}(1+\log (1+t))a^d(t).
\end{split}
\end{equation}
Combining \eqref{a lower bound}, \eqref{gamma2} and \eqref{gamma3}, we arrive at \eqref{a bound} finally.

It remains to check the lower bound of $a_M(t)$. For $\gamma\ge 1+\frac 2d$, we deduce from \eqref{H derivative}, \eqref{decomposition} and the mass conservation \eqref{mass conservation} that
\begin{equation*}
	\begin{split}
		H'(t)\le \frac {d^2}{2}\int_0^{a(t)}(2\mu+\rho)r^{d-1}dr\le C(1+a^d(t)).
	\end{split}
\end{equation*}
Integrating the above with respect to time, we get
\begin{equation*}
	(1+t)^{2}\int_0^{a(t)}\rho^\gamma r^{d-1}dr\le C\br{(1+t)+\int_0^t a^{d}(s)ds },
\end{equation*}
which together with \eqref{a lower bound} yields that
\begin{equation}
	\br{a(t)}^{-d(\gamma-1)}\le C(1+t)^{-2}\br{(1+t)+\int_0^t (a(s))^dds}.
\end{equation}
Therefore, we claim \eqref{aM bound} holds by a simple proof of contradiction.\\
 As for $\gamma\in (1,1+\frac 2d)$, it can be estimated as
\begin{equation*}
	\begin{split}
		H'(t)\le \frac{2-d(\gamma-1)}{(1+t)}H(t)+C(1+a^d(t)).
	\end{split}
\end{equation*}
Multiplying it with $(1+t)^{2-d(\gamma-1)}$, we have
\begin{equation*}
	\begin{split}
		(1+t)^{2-d(\gamma-1)}((1+t)^{-(2-d(\gamma-1))}H)'\le & C(a^d(t)+1).
	\end{split}
\end{equation*}
Consequently, for $\gamma\not=1+\frac1d$, 
\begin{equation}
	\begin{split}
		\int_0^{a(t)}\rho^\gamma r^{d-1}dr\le C(1+t)^{-1}+C(1+t)^{-d(\gamma-1)}\int_0^t (1+s)^{-(2-d(\gamma-1))}a^d(s)ds.
	\end{split}
\end{equation}
For $\gamma=1+\frac 1d$, 
\begin{equation}
	\begin{split}
		\int_0^{a(t)}\rho^\gamma r^{d-1}dr\le C(1+t)^{-1}\log (1+t)+C(1+t)^{-1}\int_0^t (1+s)^{-1}a^d(s)ds.
	\end{split}
\end{equation}
Together with \eqref{a lower bound}, and proving by contradiction,  we finally arrive at \eqref{aM bound}.
\section{The a priori estimates}
\label{sec 4}

In this section, we will derive the main a priori estimates \eqref{high integrability}, \eqref{high integrability v} and \eqref{xi estimate},  then obtain the bound \eqref{density bound} of the density along particle paths, which is crucial to establish the global existence of the strong solutions.

We define the effective viscous flux $F$ and auxiliary function $\theta$ of the density as follows
\begin{equation*}
\begin{split}
	F(r,t)=& (2\mu +\rho)\div \ub-P,\\
	\theta(r,t)=& 2\mu \ln \rho +\rho.
\end{split}
\end{equation*}
Then by \eqref{radial CNS} it satisfies that
\begin{equation}
\label{theta equation}
	\begin{split}
		\theta_t + u\partial_r \theta=-(2\mu+\rho)\div \ub=-F-P,
	\end{split}
\end{equation}
and
\begin{equation}
\label{F eqaution}
	(\rho u)_t +\partial_r \br{\rho u^2 }+\frac{\rho u^2}{r}=\partial_r F.
\end{equation}
Integrating \eqref{F eqaution} on $[a(t), r]$, we have
\begin{equation*}
	\begin{split}
		F(r,t)=&F(a(t),t)+\int_{a(t)}^r (\rho u)_t ds+\int_{a(t)}^r \partial_r (\rho u^2)ds+\int_{a(t)}^r \frac{\rho u^2}{s}ds\\
		=& \partial_t \int_{a(t)}^r \rho u ds+u \partial_r \int_{a(t)}^r \rho uds+\int_{a(t)}^r \frac{\rho u^2}{s}ds.
	\end{split}
\end{equation*}
Substitute the above into \eqref{theta equation},  we can deduce the following structure
\begin{equation}\label{4-1}
	 \partial_t \br{\theta +\xi} +u\partial_r \br{\theta+ \xi }+\int_{a(t)}^r \frac{\rho u^2}{s}ds+P(\rho)=0, 
\end{equation}
where $\xi(r,t)=\int_{a(t)}^r\rho uds.$

Now we first show the higher integrability of the density in the following lemma: 
\begin{lema}\label{lem 4.1}
	Suppose $\gamma>1$, then there exists a constant $C>0$ depends on $\gamma, T$ and initial data $(\rho_0, \ub_0,a_0)$, such that
	\begin{equation}
    \label{high integrability}
		\begin{split}
			\sup_{t\in [0,T]}\int_{\Omega_t}\rho^{2\gamma+1}\le C.
		\end{split}
	\end{equation}
\end{lema}

\begin{proof}
	We define $f=(\theta +\xi)_+$ and test the equation \eqref{4-1} with $\rho f^{2\gamma-1}$, then it comes
\begin{equation*}
	\begin{split}
		&\int_{\Omega_t} \rho f^{2\gamma-1}(\partial_t +u\partial_r )(\theta+ \xi )d\xb\\
		=& \frac 1{2\gamma} \int_{\Omega_t} \rho \frac{D}{Dt} f^{2\gamma} d\xb=\frac 1{2\gamma} \frac{d}{dt} \int_{\Omega_t} \rho f^{2\gamma} d\xb
	\end{split}
\end{equation*}
using the assumption \eqref{at}. And so we have
\begin{equation*}
	\begin{split}
		\frac 1{2\gamma} \frac{d}{dt} \int_{\Omega_t} \rho f^{2\gamma} d\xb\le & \int_{\Omega_t} \rho f^{2\gamma-1}\br{ \int^{a(t)}_r \frac{\rho u^2}{s}ds }d\xb\\
		\le & C\norm{\rho}_{L^{2\gamma+1}(\Omega_t)}^\frac 1{2\gamma} \norm{\rho^\frac 1{2\gamma} f }_{L^{2\gamma}(\Omega_t)}^{2\gamma-1}\norm{\int^{a(t)}_r \frac{\rho u^2}{s}ds }_{L^{2\gamma+1}(\Omega_t)},
	\end{split}
\end{equation*}
and here
\begin{equation*}
	\begin{split}
		\norm{\int_{a(t)}^r \frac{\rho u^2}{s}ds }_{L^{2\gamma+1}(\Omega_t)}\leq& \norm{\frac{\rho u^2}{r}}_{L^{\frac{2(2\gamma+1)}{2\gamma+3}}(\Omega_t) } \le \norm{\rho u}_{L^{2\gamma+1}(\Omega_t)}\norm{\nabla \mathbf{u}}_{L^2(\Omega_t)}\\
		\le &  \norm{\rho }_{L^{2\gamma+1}(\Omega_t)}\norm{u}_{L^\infty(\Omega_t)} \norm{\nabla \mathbf{u}}_{L^2(\Omega_t)}\le \norm{\rho }_{L^{2\gamma+1}(\Omega_t)}\norm{\nabla \mathbf{u}}_{L^2(\Omega_t)}^2.
	\end{split}
\end{equation*}
Hence we have that
\begin{equation*}
	\frac{d}{dt}\int_{\Omega_t}\rho f^{2\gamma}d\xb\le C\br{1+\int_{\Omega_t}\rho f^{2\gamma}d\xb +\int_{\Omega_t}\rho^{2\gamma+1}d\xb }\norm{\nabla \mathbf{u} }_{L^2(\Omega_t)}^2.
\end{equation*}
On the other hand, for some $C>1$,
\begin{equation}
\label{rho f}
\begin{split}
	\int_{\Omega_t} \rho^{2\gamma+1}d\xb\le &\int_{{\Omega_t}\cap\set{\rho\le C}} \rho^{2\gamma+1}d\xb+\int_{\Omega_t\cap\set{ \rho\ge C}} \rho(2\mu \ln \rho+\rho)^{2\gamma}d\xb\\
		\le & C\int_{\Omega_t}\rho d\xb+C\int_{\Omega_t\cap\set{\rho\ge C}}\rho (\theta+\xi)_+^{2\gamma}d\xb+C\int_{\Omega_t \cap \set{\rho\ge C}}\rho|\xi|^{2\gamma}d\xb\\
		=& C +C\int_{\Omega_t} \rho f^{2\gamma}d\xb+C\int_{\Omega_t} \rho|\xi |^{2\gamma}d\xb,
\end{split}
\end{equation}
and
\begin{equation*}
	\begin{split}
		\int_{\Omega_t} \rho|\xi |^{2\gamma}d\xb\le & C\norm{\rho }_{L^{\frac {2\gamma+1}{\gamma+1}}(\Omega_t) }\norm{ \xi }_{L^{2(2\gamma+1)}(\Omega_t)}^{2\gamma}\le C\norm{\rho }_{L^{\frac {2\gamma+1}{\gamma+1} }(\Omega_t)}\norm{ \nabla \xi }_{L^{\frac{2\gamma+1}{\gamma+1}}(\Omega_t)}^{2\gamma}\\
        \le &C\norm{ \rho}_{L^{\frac {2\gamma+1}{\gamma+1} }(\Omega_t)}\norm{\rho u}_{L^{\frac{2\gamma+1}{\gamma+1}}(\Omega_t)}^{2\gamma}\le C\norm{\rho }_{L^{\frac{2\gamma+1}{\gamma+1}}(\Omega_t)}\norm{\rho}_{L^{2\gamma+1}(\Omega_t)}^\gamma\norm{\rho^\frac 12 u }_{L^2(\Omega_t)}^{2\gamma}\\
        \le& C\br{\norm{\rho}_{L^1(\Omega_t)}+\norm{\rho}_{L^{2\gamma+1}(\Omega_t)} }\norm{\rho}_{L^{2\gamma+1}(\Omega_t)}^\gamma\norm{\rho^\frac 12 u }_{L^2(\Omega_t)}^{2\gamma}\\
		\le & C+C\norm{\rho}_{L^{2\gamma+1}(\Omega_t)}^{\gamma+1}\le C_\epsilon +\epsilon\int_{\Omega_t}\rho^{2\gamma+1}dx.
	\end{split}
\end{equation*}
so we have that
\begin{equation*}
	\int_{\Omega_t}\rho^{2\gamma+1}d\xb\le C\br{1+\int_{\Omega_t}\rho f^{2\gamma}d\xb }.
\end{equation*}
And we conclude:
\begin{equation*}
	\frac{d}{dt}\int_{\Omega_t}\rho f^{2\gamma}d\xb\le C\br{1+\int_{\Omega_t}\rho f^{2\gamma}d\xb}\norm{\div \mathbf{u} }_{L^2(\Omega_t)}^2
\end{equation*}
and so $\int_{\Omega_t }\rho^{2\gamma+1}d\xb\le C$.
\end{proof}
Based on previous lemma, we can show the higher integrability of the velocity as follows.
\begin{lema}\label{lem 4.2}
	Suppose $\gamma>1$, then there exists a constant $C>0$ depends on $\gamma, T$ and initial data $(\rho_0, \ub_0,a_0)$, such that
	\begin{equation}
    \label{high integrability v}
	\sup_{t\in [0,T]}\int_{\Omega_t}\rho |\ub|^{3}d\xb\le C.
	\end{equation}
\end{lema}

\begin{proof}
	Multiplying $\eqref{CNS}_2$ by $|\mathbf{u}|\mathbf{u}$ and integrating by parts over $\Omega_t$ implies
\begin{equation*}
	\begin{split}
		&\frac 13\frac{d}{dt}\int_{\Omega_t}\rho |\ub|^3 d\xb+\int_{\Omega_t}(2\mu+\rho)|\mathbf{u}|\div \mathbf{u}|^2d\xb\\
		=& -\int_{\Omega_t}(2\mu+\rho )(\div \ub )\ub\cdot \nabla|\ub|d\xb+ \int_{\Omega_t} P(\rho) \div (|\ub|\ub )d\xb.
	\end{split}
\end{equation*}
Here
\begin{equation*}
	\begin{split}
		&-\int_{\Omega_t} (2\mu+\rho )(\div \ub )\ub\cdot \nabla |\ub|d\xb\\
		=& -\int_{\Omega_t} (2\mu+\rho )\br{\div \ub }\br{\frac \ub{|\ub|}\cdot \nabla |\ub|} |\ub|d\xb
		= -\int_{\Omega_t } (2\mu+\rho )\br{u_r+\frac 1r u }u_r |\ub|d\xb\\
		=&-\int_{\Omega_t} (2\mu+\rho ) u_r^2 |\ub|d\xb-\int_{\Omega_t} (2\mu+\rho )\br{\frac 1r uu_r }|\ub|d\xb\\
		\le & -\int_{\Omega_t} (2\mu+\rho ) u_r^2 |\ub|d\xb+\frac 12 \int_{\Omega_t} (2\mu+\rho )|\div \ub |^2|\ub|d\xb.
	\end{split}
\end{equation*}
and
\begin{equation*}
	\begin{split}
		\int_{\Omega_t} P(\rho)\div (|\ub| \ub)d\xb\le &C\int_{\Omega_t}\rho^\gamma |\ub||\nabla \ub |d\xb\\
		\le & C\norm{\rho}_{L^{2\gamma}(\Omega_t)}^\gamma \norm{\ub}_{L^\infty(\Omega_t)}\norm{\nabla \ub}_{L^2(\Omega_t)}\\
		\le & C\norm{\nabla \ub}_{L^2(\Omega_t)}^2.
	\end{split}
\end{equation*}
Combining together, we have that
\begin{equation*}
	\frac{d}{dt}\int_{\Omega_t}\rho |\ub |^3d\xb\le C\norm{\nabla \ub}_{L^2(\Omega_t)}^2.
\end{equation*}
This yields a bound for $\int_{\Omega_t}\rho |\ub|^3d\xb$.
\end{proof}
Having the above lemmas on hand, we come to the most important part of this paper, which gives the lower and upper bound of the density along the particle paths.
\begin{lema}\label{lem 4.3}
	Let $(\rho ,\ub, a(t))$ be a strong solution of free-boundary value problem \eqref{radial CNS}-\eqref{at}, then there exist a constant $C>0$ depends on $\gamma, (\rho_0, \ub_0,a_0)$ and time $T$, such that along the particle path $r=r_{x_0}(t)$ defined by
	\begin{equation*}
		\begin{dcases}
			\frac{d}{dt}r_{x_0}(t)=u(r_{x_0}(t),t),\\
			r_{x_0}(0)=r_0,x_0=1-\int_{r_0}^{a_0}\rho_0(r)rdr,
		\end{dcases}
	\end{equation*}
	we have upper bound and lower bound for density:
	\begin{equation}
	\label{density bound}
		\rho_0(r_0)e^{-C}\le \rho(r_{x_0}(t),t)\le \rho_0(r_0)e^{C}.
	\end{equation}
\end{lema}
\begin{proof}
	Consider the particle path $r=r_{x_0}(t)$ and \eqref{4-1}, we have
\begin{equation*}
	\begin{split}
	\frac{d}{dt}(\theta+\xi)(r_{x_0}(t),t)+P(\rho)(r_{x_0}(t),t)+\int_{a(t)}^{r_{x_0}(t)}\frac{\rho u^2}{s}ds=0,
	\end{split}
\end{equation*}
which yields that:
\begin{equation}\label{4-6}
	\begin{split}
		&2\mu \log \frac{\rho(r_{x_0}(t),t)}{\rho_0(r_0)}+\rho(r_{x_0}(t),t)-\rho_0(r_0)+\xi(r_{x_0}(t),t)-\xi(r_0)+\int_0^tP(\rho)d\tau\\
		 = &  \int_0^t \br{\int^{a(\tau)}_r \frac{\rho u^2}{s}ds} d\tau
		\leq C\int_0^t \norm{\rho(\tau)}_{L^\infty(\Omega_\tau)}\norm{\nabla \mathbf{u}}_{L^2(\Omega_\tau)}^2d\tau.
	\end{split}
\end{equation}
Next, by using Lemma \ref{lem 4.2} and the bound \eqref{at lower bound} of $a(t)$, one obtains
\begin{equation}
\label{xi estimate}
	\begin{split}
		|\xi(r_{x_0}(t),t)|\le \int^{a(t)}_0 \abs{\rho u }dr \le & C\int_0^{a(t)} (\rho |u|^{3} r )^{\frac 1{3} }\cdot \rho^{\frac{2}{3} } \cdot \br{r^{-\frac{1}{2}}}^\frac{2}{3}dr\\
		\le & C\norm{\rho(t)}_{L^\infty(\Omega_t)}^{\frac{2}{3}} \br{\int_0^{a(t)} \rho |u|^{3}rdr }^{\frac{1}{3} }\br{\int_0^{a(t)} r^{-\frac 1{2}}dr }^{\frac{2}{3}}\\
		\le & C \norm{\rho(t)}_{L^\infty(\Omega_t)}^{\frac{2}{3}}.
	\end{split}
\end{equation}
If $\log \frac{\rho(r_{x_0}(t),t)}{\rho_0(r_0)}<0$, then
\begin{align*}
    \rho(r_{x_0}(t),t)<\rho_0(r_0)\leq \norm{\rho_0}_{L^\infty(\Omega_0)}.
\end{align*}
If $\log \frac{\rho(r_{x_0}(t),t)}{\rho_0(r_0)}\geq0$, by using \eqref{4-6} we have that
\begin{equation*}
	\begin{split}
		\rho(r_{x_0}(t),t)
		\le &C+\norm{\rho(t)}_{L^\infty(\Omega_t)}^{\frac{2}{3}}+C\int_0^t \norm{\rho(\tau)}_{L^\infty(\Omega_\tau)}\norm{\nabla \mathbf{u}}_{L^2(\Omega_\tau)}^2d\tau\\
		\le &C+\frac 12 \norm{\rho(t)}_{L^\infty(\Omega_t)}+C\int_0^t \norm{\rho(\tau)}_{L^\infty(\Omega_\tau)}\norm{\nabla \mathbf{u}}_{L^2(\Omega_\tau)}^2d\tau.
	\end{split}
\end{equation*}
Therefore, take supreme on the left hand, we get
\begin{equation*}
	\begin{split}
		\norm{\rho(t)}_{L^\infty(\Omega_t)}\le C +C\int_0^t \norm{\rho(\tau)}_{L^\infty(\Omega_\tau)}\norm{\nabla \mathbf{u}}_{L^2(\Omega_\tau)}^2d\tau.
	\end{split}
\end{equation*}
Use the Gr\"onwall inequality, we obtain that 
$\norm{\rho(t)}_{L^\infty(\Omega_t)}\le C,$ and take supreme about $t\in[0,T]$, we get
$$\sup_{0\leq t\leq T}\norm{\rho(t)}_{L^\infty(\Omega_t)}\le C$$
Together with \eqref{4-6} , it comes that
\begin{equation*}
	2\mu \log \frac{\rho(r_{x_0}(t),t)}{\rho_0(r_0)}\le C,
\end{equation*}
which implies $$\rho(r_{x_0}(t),t)\le \rho_0(r_0)e^C.$$

Applying this bound to \eqref{4-6} again, we can obtain the lower bound
$$\rho(r_{x_0}(t),t)\ge \rho_0(r_0)e^{-C}.$$
The proof is now finished.
\end{proof}

\section{Proof of Theorem \ref{global existence}}
\label{sec 5}
In fact, we can obtain the global existence of strong solutions through Lemma \ref{lem 4.3}, and the higher-order estimates for $\rho$ and $u$ are the same as in \cite{2016 Li-Zhang-JDE}. But for the completeness of the article, we still follow the idea of \cite{2016 Li-Zhang-JDE} to obtain higher-order estimates.

We rewrite the equations \eqref{radial CNS} in Lagrangian mass coordinate:
\begin{align*}
x(r,t)=\int_0^r\rho sds, \quad \tau(r,t)=t.
\end{align*}
The transformation turns $[0,a(t)]\times[0,T]$ into $[0,1]\times[0,T]$ and satisfies
\begin{align*}
\frac{\partial x}{\partial r}=\rho r,\quad \frac{\partial x}{\partial t}=-\rho ur,\quad \frac{\partial \tau}{\partial r}=0, \quad \frac{\partial \tau}{\partial t}=1.
\end{align*}
Then the free boundary problem \eqref{radial CNS}-\eqref{at} is transformed into the following fixed boundary problem:
\begin{align}\label{5-1}
\left\{ 
\begin{array}{lc}
\rho_\tau+\rho^2(ru)_x=0,\\
\frac{1}{r}u_\tau+(\rho^\gamma-(2\mu+\rho)\rho(ru)_x)_x=0,
\end{array}
\right.
\end{align}
which $(x,\tau)\in[0,1]\times[0,T]$, with the initial data and boundary conditions 
\begin{align*}
(\rho, u)(x,0)=(\rho_0,u_0),\quad x\in[0,1],\\
u(0,\tau)=0,\quad (\rho^\gamma-(2\mu+\rho)\rho(ru)_x)(0,\tau)=0,\quad \tau\in[0,T],
\end{align*}
where $r=r(x,\tau)$ is defined by
\begin{align*}
\frac{d}{d\tau}r(x,\tau)=u(x,\tau), \quad x\in[0,1], \tau\in[0,T].
\end{align*}
The energy estimate in Lagrangian coordinate is that
\begin{align*}
\sup_{0\leq \tau\leq T}\int_0^1\left(|u|^2+\dfrac{1}{\gamma-1}\rho^{\gamma-1}\right)dx+\int_0^T\int_0^1(2\mu+\rho)\rho|(ru)_x|^2dxd\tau\leq E_0.
\end{align*}
By H\"{o}lder inequality and the standard energy estimate, we get
\begin{align*}
1=\int_{\Omega_t}\rho dx\leq \left(\int_{\Omega_t}\rho^\gamma dx\right)^{\frac{1}{\gamma}}|\Omega_{t}|^{1-\frac{1}{\gamma}}\leq C(\gamma-1)E_0 a(t)^{2-\frac{2}{\gamma}},
\end{align*}
which implies there exists a constant $\alpha>0$ depending only on $E_0$ and $\gamma$ such that $\forall t\in[0,T]$, $a(t)\geq 3\alpha$.

For the high-order estimation below, we need to introduce cut-off functions $\zeta$ and $\chi$. Let $\zeta\in C^\infty$ be a smooth function of $r$ satisfying
\begin{align*}
\begin{split}
0\leq \zeta\leq 1; \quad \zeta=1, r\in[0,\alpha]; \quad \text{supp}(\zeta)\subset [0,2\alpha];\\
\zeta=0, r\in[2\alpha,a(t)], \quad \forall t\in[0,T],\quad |\zeta'|\leq\frac{2}{\alpha} \quad \text{and} \quad|\zeta''|\leq \frac{10}{\alpha^2}.
\end{split}
\end{align*}
Due to Lemma \ref{lem 4.3}, there exists a constant $x_1$ depending only on $T,\gamma,E_0$ such that 
\begin{align*}
0<x_1\leq \int_0^\alpha \rho rdr<1, \quad 0\leq t\leq T.
\end{align*}
Also, we define a smooth function $\chi$ of Lagrangian coordinates $x$ such that
\begin{align*}
\begin{split}
0\leq \chi\leq 1; \quad \chi=1, x\in[x_1,1]; \quad \text{supp}(\chi)\subset [x_0,1];\\
\chi=0, x\in[0,x_0], \quad \forall \tau\in[0,T],\quad |\chi'|\leq\frac{2}{x_1-x_0},
\end{split}
\end{align*}
which $0<x_0<x_1<1.$

\begin{prop}
	There exists a constant $C$ depending on $\gamma,T,x_0$ and initial data such that the following estimate holds:
	\begin{equation}
	\label{5-2}
		\sup_{t\in [0,T]}\int_{\Omega_t}F^2 d\xb +\int_0^T \int_{\Omega_t} \rho |\dot \ub |^2 d\xb dt\le C,
	\end{equation}
\end{prop}

\begin{proof}
We start with the boundary estimates in the Lagrangian coordinates. Multiplying $\eqref{5-1}_2$ by $ru_\tau\chi^2$ and integrating $x$ over $[0,1]$, we have
\begin{align}\label{5-3}
\int_0^1u_\tau^2\chi^2dx+\int_0^1F(ru_\tau)_x\chi^2dx=-2\int_0^1Fru_\tau \chi'\chi dx.
\end{align}
Note that $F=(2\mu+\rho)\rho(ru)_x-\rho^\gamma$ and $(ru)_x=\dfrac{F+\rho^\gamma}{(2\mu+\rho)\rho},$ we calculate
\begin{align*}
\begin{split}
(ru_\tau)_x&=(ru)_{\tau x}-2uu_x\\
&=\dfrac{F_\tau}{(2\mu+\rho)\rho}+\dfrac{(\rho^\gamma)_\tau}{(2\mu+\rho)\rho}+\left(\dfrac{1}{(2\mu+\rho)\rho}\right)_\tau (F+\rho^\gamma)-2uu_x.
\end{split}
\end{align*}
Substituting the above equality into \eqref{5-3}, we get
\begin{align}\label{5-4}
\begin{split}
&\quad\int_0^1u_\tau^2\chi^2dx+\dfrac{1}{2}\dfrac{d}{d\tau}\int_0^1\dfrac{F^2}{(2\mu+\rho)\rho}\chi^2dx\\
&=-2\int_0^1Fru_\tau \chi'\chi dx-\frac{1}{2}\int_0^1 F^2\left(\dfrac{1}{(2\mu+\rho)\rho}\right)_\tau\chi^2dx\\
&\quad-\int_0^1F\left(\left(\dfrac{1}{(2\mu+\rho)\rho}\right)_\tau\rho^\gamma+\dfrac{(\rho^\gamma)_\tau}{(2\mu+\rho)\rho}\right)\chi^2dx+2\int_0^1Fuu_x\chi^2\\
&:=I_1+I_2+I_3+I_4.
\end{split}
\end{align}
For the first term, we have
\begin{align}\label{5-5}
\begin{split}
|I_1|&\leq C\norm{u_\tau\chi}_{L^2(0,1)}\norm{Fr\chi'}_{L^2(0,1)}\\
&\leq \dfrac{1}{4}\norm{u_\tau\chi}_{L^2(0,1)}^2+C\norm{\sqrt{\rho}(ru)_x}_{L^2(0,1)}^2+C.
\end{split}
\end{align}
Next, by the Gagliardo-Nirenberg inequality and \eqref{5-1}, we get
\begin{align}\label{5-6}
\begin{split}
|I_2|&\leq C\int_0^1\frac{F^2}{\sqrt{\rho}}\sqrt{\rho}(ru)_x\chi^2dx\\
&\leq C\norm{F\chi}_{L^2(0,1)}^{\frac{1}{2}}\left(\norm{F\chi'}_{L^2(0,1)}^{\frac{1}{2}}+\norm{F_x\chi}_{L^2(0,1)}^{\frac{1}{2}}\right)\norm{\frac{F\chi}{\sqrt{\rho}}}_{L^2(0,1)}\norm{\sqrt{\rho}(ru)_x}_{L^2(0,1)}\\
&\leq \dfrac{1}{4}\norm{u_\tau\chi}_{L^2(0,1)}^2+C\left(1+\norm{\frac{F\chi}{\sqrt{\rho}}}_{L^2(0,1)}^2\right)\left(1+\norm{\sqrt{\rho}(ru)_x}_{L^2(0,1)}^2\right).
\end{split}
\end{align}
Next, the H\"{o}lder inequality yields that
\begin{align}\label{5-7}
\begin{split}
|I_3|&\leq C\int_0^1\frac{|F|}{\sqrt{\rho}}\sqrt{\rho}(ru)_x\chi^2dx\leq C\norm{\frac{F\chi}{\sqrt{\rho}}}_{L^2(0,1)}\norm{\sqrt{\rho}(ru)_x}_{L^2(0,1)}\\
&\leq C\norm{\frac{F\chi}{\sqrt{\rho}}}_{L^2(0,1)}^2\norm{\sqrt{\rho}(ru)_x}_{L^2(0,1)}^2+C.
\end{split}
\end{align}
Finally, note that $\norm{u}_{L^\infty(0,a(t))}\leq\norm{\div \mathbf{u}}_{L^2(\Omega_{t})}=\norm{\sqrt{\rho}(ru)_x}_{L^2(0,1)},$ we get
\begin{align}\label{5-8}
\begin{split}
|I_4|&\leq C\left|\int_0^1\frac{1}{r} u(ru)_xF\chi^2dx\right|+C\left|\int_0^1\frac{r_x}{r}u^2F\chi^2dx\right|\\
&\leq C\norm{u}_{L^\infty(0,1)}\norm{\frac{F\chi}{\sqrt{\rho}}}_{L^2(0,1)}\norm{\sqrt{\rho}(ru)_x}_{L^2(0,1)}+C\left|\int_0^1 u^2\frac{1}{\rho r^2}F\chi^2dx\right|\\
&\leq C\norm{u}_{L^\infty(0,1)}\norm{\frac{F\chi}{\sqrt{\rho}}}_{L^2(0,1)}\norm{\sqrt{\rho}(ru)_x}_{L^2(0,1)}+C\norm{\frac{F\chi}{\sqrt{\rho}}}_{L^2(0,1)}\norm{u}_{L^\infty(0,1)}^2\norm{\frac{1}{\sqrt{\rho}}}_{L^2(0,1)}\\
&\leq C\left(1+\norm{\frac{F\chi}{\sqrt{\rho}}}_{L^2(0,1)}^2\right)\left(1+\norm{\sqrt{\rho}(ru)_x}_{L^2(0,1)}^2\right),
\end{split}
\end{align}
where we used the simple fact that if $x \geq x_0$, then there exists a constant $r_0>0$ such that $r\geq r_0$.

Substituting the all above estimates \eqref{5-5}-\eqref{5-8} into \eqref{5-4}, we get
\begin{align*}
\begin{split}
  &\quad\int_0^1u_\tau^2\chi^2dx+\dfrac{1}{2}\dfrac{d}{d\tau}\int_0^1\dfrac{F^2}{(2\mu+\rho)\rho}\chi^2dx \\
   &\leq C\left(1+\norm{\frac{F\chi}{\sqrt{(2\mu+\rho)\rho}}}_{L^2(0,1)}^2\right)\left(1+\norm{\sqrt{\rho}(ru)_x}_{L^2(0,1)}^2\right).  
\end{split}
\end{align*}
Using Gronwall's inequality, we get
\begin{align}\label{5-9}
\sup_{0\leq \tau\leq T}\int_0^1\dfrac{F^2}{(2\mu+\rho)\rho}\chi^2dx+\int_0^T\int_0^1u_\tau^2\chi^2dxd\tau\leq C.
\end{align}

Next, we deal with the interior estimates in Eulerian coordinates.

Multiplying the momentum equation $\eqref{CNS}_
2$ by $\zeta^2\dot{\mathbf{u}}$ and integrating the resulting equality over $\Omega_t$, we have 
\begin{align}\label{5-10}
\begin{split}
\int_{\Omega_t}\rho|\dot{\mathbf{u}}|^2\zeta^2dx=\int_{\Omega_t}\nabla F\cdot \zeta^2\dot{\mathbf{u}}dx=-\int_{\Omega_t} F\div \dot{\mathbf{u}}\zeta^2dx-\int_{\Omega_t} F\nabla(\zeta^2)\cdot\dot{\mathbf{u}}dx,
\end{split}
\end{align}
which we used the stress-free boundary condition \eqref{stress free condition}.
 Note that 
\begin{align*}
\begin{split}
\div \dot{\mathbf{u}}&=\frac{D}{Dt}\div \mathbf{u}+|\div \mathbf{u}|^2-2\nabla u_1\cdot\nabla^\perp u_2\\
&= \dfrac{D}{Dt}\left(\dfrac{F+P}{2\mu+\rho}\right)+|\div \mathbf{u}|^2-2\nabla u_1\cdot\nabla^\perp u_2
\end{split}
\end{align*}
Then substituting the above equality into \eqref{5-10}, we get
\begin{align}\label{5-11}
\begin{split}
&\quad\int_{\Omega_t}\rho|\dot{\mathbf{u}}|^2\zeta^2dx+\frac{1}{2}\frac{d}{dt}\int_{\Omega_t}\frac{F^2}{2\mu+\rho}\zeta^2dx\\
&=\dfrac{1}{2}\int_{\Omega_t}F^2\div \mathbf{u}\left(\rho\left(\dfrac{1}{2\mu+\rho}\right)'-\dfrac{1}{2\mu+\rho}\right)\zeta^2dx\\
&\quad-\int_{\Omega_t}F\div \mathbf{u}\left(\rho\left(\dfrac{\rho^\gamma}{2\mu+\rho}\right)'-\dfrac{\rho^\gamma}{2\mu+\rho}\right)\zeta^2dx\\
&\quad-2\int_{\Omega_t}F\nabla u_1\cdot\nabla^\perp u_2\zeta^2dx-\int_{\Omega_t}F\zeta\nabla\zeta\cdot\dot{\mathbf{u}}dx\\
&\leq C\int_{\Omega_t} F^2|\div\mathbf{u}|\zeta^2dx+C\int_{\Omega_t} |F||\div\mathbf{u}|\zeta^2dx\\
&\quad+C\left|\int_{\Omega_t}F\nabla u_1\cdot\nabla^\perp u_2\zeta^2dx\right|+C\left|\int_{\Omega_t}F\zeta\nabla\zeta\cdot\dot{\mathbf{u}}dx\right|\\
&:=J_1+J_2+J_3+J_4.
\end{split}
\end{align}
For the first term,
\begin{align*}
\begin{split}
|J_1|&\leq C\norm{\div \mathbf{u}}_{L^2(\Omega_t)}\norm{F\zeta}_{L^4(\Omega_t)}^2\\
&\leq C\norm{\div \mathbf{u}}_{L^2(\Omega_t)}\norm{F\zeta}_{L^2(\Omega_t)}\norm{\nabla (F\zeta)}_{L^2(\Omega_t)}\\
&\leq \frac{1}{2}\norm{\sqrt{\rho}\dot{\mathbf{u}}\zeta}_{L^2(\Omega_{t})}^2+C\norm{\div \mathbf{u}}_{L^2(\Omega_t)}^2\norm{F\zeta}_{L^2(\Omega_t)}^2\\
&\quad+C\norm{\div \mathbf{u}}_{L^2(\Omega_t)}\norm{F\zeta}_{L^2(\Omega_t)}\norm{F\nabla \zeta}_{L^2(\Omega_t)}\\
&\leq \frac{1}{2}\norm{\sqrt{\rho}\dot{\mathbf{u}}\zeta}_{L^2(\Omega_{t})}^2+C\norm{\div \mathbf{u}}_{L^2(\Omega_t)}^2\norm{\frac{F\zeta}{\sqrt{2\mu+\rho}}}_{L^2(\Omega_t)}^2+C,
\end{split}
\end{align*}
where we used $\norm{F\nabla \zeta}_{L^2(\Omega_t)}\leq C$, it is because \eqref{5-9} and $\text{supp}\{\nabla\zeta\}\subset B_{2\alpha}-B_{\alpha}$.\\
Next, the H\"{o}lder inequality yields that
\begin{align*}
|J_2|\leq C \norm{\div \mathbf{u}}_{L^2(\Omega_t)}\norm{\frac{F\zeta}{\sqrt{2\mu+\rho}}}_{L^2(\Omega_t)}.
\end{align*}
Next, the third term can be estimated as
\begin{align*}
\begin{split}
|J_3|&\leq\left|\int_{\Omega_t}F\nabla (\zeta u_1)\cdot\nabla^\perp (\zeta u_2)dx\right|+\left|\int_{\Omega_t}F(\nabla u_1\cdot\nabla^\perp u_2)\zeta^2-\nabla (\zeta u_1)\cdot\nabla^\perp (\zeta u_2))dx\right|\\
&\leq \left|\int_{\mathbb{R}^2}F\zeta\nabla (\zeta u_1)\cdot\nabla^\perp (\zeta u_2)dx\right|+\left|\int_{\mathbb{R}^2}F(1-\zeta)\nabla (\zeta u_1)\cdot\nabla^\perp (\zeta u_2)dx\right|\\
&\quad +\left|\int_{\Omega_t}F\zeta\nabla\zeta\cdot x\dfrac{u^2}{r^2}dx\right|.\\
\end{split}
\end{align*}
Computing that
\begin{align*}
\begin{split}
\left|\int_{\mathbb{R}^2}F\zeta\nabla (\zeta u_1)\cdot\nabla^\perp (\zeta u_2)dx\right|&\leq C\norm{F\zeta}_{BMO}\norm{\nabla (\zeta u_1)\cdot\nabla^\perp (\zeta u_2)}_{\mathcal{H}^1}\\
&\leq C\norm{\nabla(F\zeta)}_{L^2(\Omega_{t})}\norm{\nabla(\zeta\mathbf{u})}_{L^2(\Omega_{t})}^2\\
&\leq C(\norm{\sqrt{\rho}\dot{\mathbf{u}}\zeta}_{L^2(\Omega_{t})}+1)(\norm{\zeta\div\mathbf{u}}_{L^2(\Omega_{t})}^2+\norm{\mathbf{u}\nabla\zeta}_{L^2(\Omega_{t})}^2)\\
&\leq \dfrac{1}{4}\norm{\sqrt{\rho}\dot{\mathbf{u}}\zeta}_{L^2(\Omega_{t})}^2+C\norm{\div\mathbf{u}}_{L^2(\Omega_{t})}^4+C\\
&\leq \dfrac{1}{4}\norm{\sqrt{\rho}\dot{\mathbf{u}}\zeta}_{L^2(\Omega_{t})}^2+C\norm{\div \mathbf{u}}_{L^2(\Omega_t)}^2\left(\norm{\frac{F\zeta}{\sqrt{2\mu+\rho}}}_{L^2(\Omega_t)}^2+1\right)+C,
\end{split}
\end{align*}
which we used $\zeta+\chi\geq 1$ and \eqref{5-9},\\
\begin{align*}
\begin{split}
&\quad\left|\int_{\mathbb{R}^2}F(1-\zeta)\nabla (\zeta u_1)\cdot\nabla^\perp (\zeta u_2)dx\right|\\
&\leq C\int_{B_{2\alpha}-B_{\alpha}}|F|\left|\partial_r(\zeta u)\dfrac{\zeta u}{r}\right|dx\\
&\leq C\int_{B_{2\alpha}-B_{\alpha}}|F|\zeta u^2dx+C\int_{B_{2\alpha}-B_{\alpha}} |F|\zeta^2|u\partial_ru|dx\\
&\leq C\norm{\div \mathbf{u}}_{L^2(\Omega_t)}^2\norm{\frac{F\zeta}{\sqrt{2\mu+\rho}}}_{L^2(\Omega_t)}\\
&\leq C\norm{\div \mathbf{u}}_{L^2(\Omega_t)}^2\left(\norm{\frac{F\zeta}{\sqrt{2\mu+\rho}}}_{L^2(\Omega_t)}^2+1\right),
\end{split}
\end{align*}
and
\begin{align*}
\left|\int_{\Omega_t}F\zeta\nabla\zeta\cdot x\dfrac{u^2}{r^2}dx\right|\leq  C\norm{\div \mathbf{u}}_{L^2(\Omega_t)}^2.
\end{align*}
Thus
\begin{align*}
|J_3|\leq \frac{1}{4}\norm{\sqrt{\rho}\dot{\mathbf{u}}\zeta}_{L^2(\Omega_{t})}^2+C\left(\norm{\div \mathbf{u}}_{L^2(\Omega_t)}^2+1\right)\left(\norm{\frac{F\zeta}{\sqrt{2\mu+\rho}}}_{L^2(\Omega_t)}^2+1\right).
\end{align*}
Then substituting the above equality into \eqref{5-11}, we get
\begin{align*}
\begin{split}
&\quad\frac{1}{2}\frac{d}{dt}\int_{\Omega_t}\frac{F^2}{2\mu+\rho}\zeta^2dx+\int_{\Omega_t}\rho|\dot{\mathbf{u}}|^2\zeta^2dx\\
&\leq C\left(\norm{\div \mathbf{u}}_{L^2(\Omega_t)}^2+1\right)\left(\norm{\frac{F\zeta}{\sqrt{2\mu+\rho}}}_{L^2(\Omega_t)}^2+1\right)+\norm{\dot{\mathbf{u}}\nabla\zeta}_{L^2(\Omega_{t})}^2
\end{split}\label{5-12}
\end{align*}
By Gronwall's inequality, we obtain
\begin{align}
\sup_{0\leq t\leq T}\frac{1}{2}\frac{d}{dt}\int_{\Omega_t}\frac{F^2}{2\mu+\rho}\zeta^2dx+\int_0^T\int_{\Omega_t}\rho|\dot{\mathbf{u}}|^2\zeta^2dxdt\leq C,
\end{align}
which together with \eqref{5-9} complete the proof.
\end{proof}

\begin{prop}
	There exists a constant $C$ depending on $\gamma,T,\xb_0$ and initial data such that the following estimate holds:
	\begin{equation}
	\label{U estimate}
		\sup_{t\in [0,T]}\int_{\Omega_t}\rho |\dot \ub |^2 d\xb +\int_0^T \int_{\Omega_t} (2\mu+\rho)|\div \dot \ub |^2 d\xb dt\le C.
	\end{equation}
\end{prop}

\begin{proof}
Differentiating the equation $\eqref{5-1}_2$ with respect to $\tau$, we obtain
\begin{align*}
\frac{1}{r}u_{\tau \tau}-\frac{1}{r^2}u_\tau u+(\rho^\gamma-(2\mu+\rho)\rho(ru)_x)_{x\tau}=0
\end{align*}
Multiplying the above equation by $ru_\tau\chi^2$ and integrating $x$ over $[0,1]$, we have
\begin{align}
\begin{split}
&\quad \frac{1}{2}\frac{d}{d\tau}\int_0^1u_\tau^2\chi^2dx+\int_0^1((2\mu+\rho)\rho(ru)_x-\rho^\gamma)_\tau\chi^2u_{x\tau}rdx\\
&=\int_0^1\frac{1}{r}u_\tau^2u\chi^2dx-\int_0^1((2\mu+\rho)\rho(ru)_x-\rho^\gamma)_\tau(2\chi\chi'u_\tau r+u_\tau\chi^2r_x)dx.
\end{split}
\end{align}
Combining the following facts
\begin{align*}
\begin{split}
\norm{u}_{L^\infty(0,1)}&\leq C\norm{\div \mathbf{u}}_{L^2(\Omega_t)}\leq C\norm{F}_{L^2(\Omega_t)}\leq C\\
\norm{\rho_\tau}_{L^2(0,1)}&=\norm{\rho^2(ru)_x}_{L^2(0,1)}\leq C\\
\norm{u_\tau \chi}_{L^\infty(0,1)}&\leq C\norm{u_\tau \chi}_{L^1(0,1)}+C\norm{u_{\tau x}\chi+u_\tau\chi'}_{L^1(0,1)}\\
&\leq C\norm{u_\tau}_{L^2(0,1)}+C\norm{\sqrt{\rho}u_{x\tau}\chi}_{L^2(0,1)}\\
\norm{F\chi}_{L^\infty(0,1)}
&\leq C\norm{F\chi}_{L^1(0,1)}+C\norm{F_x\chi+F\chi'}_{L^1(0,1)}\\
&\leq C\norm{F}_{L^2(0,1)}+C\norm{u_\tau\chi}_{L^2((0,1))},
\end{split}
\end{align*}
we can obtain 
\begin{align*}
\begin{split}
&\quad\frac{d}{d\tau}\int_0^1(u_\tau^2+\gamma\rho^{\gamma+1})|(ru)_x|^2)\chi^2dx+\int_0^1(2\mu+\rho)\rho u_{x\tau}^2r\chi^2dx\\
&\leq C\left(1+\norm{u_\tau}_{L^2(0,1)}^2+\norm{F}_{L^2(0,1)}^2\right).
\end{split}
\end{align*}
By Gronwall's inequality, we obtain
\begin{align}\label{5-15}
&\quad\sup_{0\leq \tau\leq T}\int_0^1(u_\tau^2+\gamma\rho^{\gamma+1}|(ru)_x|^2)\chi^2dx+\int_0^T\int_0^1(2\mu+\rho)\rho u_{x\tau}^2r\chi^2dxd\tau\leq C.
\end{align}

Next, we deal with the interior estimate in Eulerian coordinates. Operating $\zeta^2\dot{\mathbf{u}}_j[\partial/\partial_t+\div(\mathbf{u}\cdot)]$ to both side of $\eqref{CNS}_2$ and integrating the resulted equations by parts, we have
\begin{align*}
\begin{split}
&\quad\dfrac{1}{2}\dfrac{d}{dt}\int_{\Omega_t}\rho|\dot{\mathbf{u}}|^2\zeta^2dx+\int_{\Omega_t}(2\mu+\rho)|\div\dot{\mathbf{u}}|^2\zeta^2dx\\
&\leq C+C\norm{\nabla\mathbf{u}}_{L^4(\Omega_t)}^4+C\norm{\nabla\mathbf{u}}_{L^2(\Omega_t)}^2+\frac{\mu}{2}\norm{\zeta\div\dot{\mathbf{u}}}_{L^2(\Omega_t)}^2\\
&\leq C+C\int_{\Omega_t}\rho|\dot{\mathbf{u}}|^2\zeta^2dx+C\norm{\div\mathbf{u}}_{L^2(\Omega_t)}^2+\frac{\mu}{2}\norm{\zeta\div\dot{\mathbf{u}}}_{L^2(\Omega_t)}^2,
\end{split}
\end{align*}
where we used
\begin{align*}
\begin{split}
\norm{\nabla\mathbf{u}}_{L^4(\Omega_t)}^4&\leq \norm{\nabla(\zeta\mathbf{u})}_{L^4(\Omega_t)}^4+\norm{(1-\zeta)\nabla\mathbf{u}-\nabla\zeta\mathbf{u}}_{L^4(\Omega_t)}^4\\
&\leq C+\norm{\div (\zeta\mathbf{u})}_{L^4(\Omega_t)}^4+\norm{(1-\zeta)F}_{L^4(\Omega_t)}^4\\
&\leq C+\norm{\zeta\div \mathbf{u}}_{L^4(\Omega_t)}^4\leq C+\norm{\zeta F}_{L^4(\Omega_t)}^4\\
&\leq C+\norm{\zeta F}_{L^2(\Omega_t)}^2\norm{\nabla(\zeta F)}_{L^2(\Omega_t)}^2\\
&\leq C+C\norm{\zeta\nabla F}_{L^2(\Omega_t)}^2\leq C+C\int_{\Omega_t}\rho|\dot{\mathbf{u}}|^2\zeta^2dx.
\end{split}
\end{align*}
Then by Gronwall's inequality, we obtain 
\begin{align}\label{5-16}
\sup_{0\leq t\leq T}\int_{\Omega_t}\rho|\dot{\mathbf{u}}|^2\zeta^2dx+\int_0^T\int_{\Omega_t}(2\mu+\rho)|\div\dot{\mathbf{u}}|^2\zeta^2dxdt\leq C.
\end{align}
We complete the proof with \eqref{5-15} and \eqref{5-16}.
\end{proof}

In the final proposition, we will show the higher order derivative estimates of the density and velocity, which guarantee the local strong solution can be extended to a global one.
\begin{prop}
For any $p>2$, There exists a constant $C$ depending on $\gamma,T,\xb_0$ and initial data such that the following estimate holds:
\begin{equation}
\label{rho estimate}
	\sup_{t\in [0,T]}\br{ \norm{\rho}_{L^1(\Omega_t)\cap W^{1,p}(\Omega_t)} +\norm{\ub}_{H^2(\Omega_t)} }+\int_0^T \br{\norm{\nabla \ub }_{L^\infty(\Omega_t)}^2+\norm{\ub_t}_{H^1(\Omega_t)}^2 }dt\le C.
\end{equation}
\end{prop}
\begin{proof}
Operating $\nabla$ to the mass equation $\eqref{CNS}_1$, multiplying the resulting equation by $q|\nabla\rho|^{p-2}\nabla\rho$, we get
\begin{align}\label{5-18}
\begin{split}
&(|\nabla\rho|^p)_t+\div(|\nabla\rho|^p\mathbf{u})+(p-1)|\nabla\rho|^p\div\mathbf{u}+p|\nabla\rho|^{p-2}\nabla\rho\cdot(\nabla\mathbf{u}\cdot\nabla\rho)\\
&\quad+p\rho|\nabla\rho|^{p-2}\nabla\rho\cdot\nabla\div\mathbf{u}=0.
\end{split}
\end{align}
Note that
\begin{align*}
\frac{d}{dt}\int_{\Omega_t}|\nabla\rho|^pdx=\int_{\Omega_t}\div(|\nabla\rho|^p\mathbf{u})dx+\int_{\Omega_t}\partial_t(|\nabla\rho|^p)dx.
\end{align*}
Integrating the equation \eqref{5-18} over $\Omega_t$ we get
\begin{align*}
\begin{split}
\dfrac{d}{dt}\norm{\nabla\rho}_{L^p(\Omega_t)}&\leq C\left( \norm{\nabla\rho}_{L^p(\Omega_t)}\norm{\div \mathbf{u}}_{L^\infty(\Omega_t)}+\norm{\nabla\rho}_{L^p(\Omega_t)}\norm{\nabla \mathbf{u}}_{L^\infty(\Omega_t)}\right.\\
&\quad+ \left.\norm{\nabla\rho}_{L^p(\Omega_t)}+\norm{\nabla\div\mathbf{u}}_{L^p(\Omega_t)}\right)\\
&\leq C\left(1+\norm{F}_{L^\infty(\Omega_t)}+\norm{\nabla\mathbf{u}}_{L^\infty(\Omega_t)}\right)\norm{\nabla\rho}_{L^p(\Omega_t)}+C\norm{\rho\dot{\mathbf{u}}}_{L^p(\Omega_t)}\\
&\leq 
C\left(1+\norm{F}_{L^2(\Omega_t)}^{\frac{p-2}{2p-2}}\norm{\nabla F}_{L^p(\Omega_t)}^{\frac{p}{2p-2}}+\norm{\nabla\mathbf{u}}_{L^\infty(\Omega_t)}\right)\norm{\nabla\rho}_{L^p(\Omega_t)}+C\norm{\rho\dot{\mathbf{u}}}_{L^p(\Omega_t)}\\
&\leq 
C\left(1+\norm{\dot{\mathbf{u}}}_{L^p(\Omega_t)}^{\frac{p}{2p-2}}+\norm{\nabla\mathbf{u}}_{L^\infty(\Omega_t)}\right)\norm{\nabla\rho}_{L^p(\Omega_t)}+C\norm{\rho\dot{\mathbf{u}}}_{L^p(\Omega_t)}\\
&\leq C\left(1+\norm{\div\dot{\mathbf{u}}}_{L^2(\Omega_t)}+\norm{\nabla\mathbf{u}}_{L^\infty(\Omega_t)}\right)\norm{\nabla\rho}_{L^p(\Omega_t)}+C\norm{\rho\dot{\mathbf{u}}}_{L^p(\Omega_t)},
\end{split}
\end{align*}
which we used Lemma \ref{lem 2.4} to obtain that
\begin{align*}
\begin{split}
\norm{\dot{\mathbf{u}}}_{L^p(\Omega_t)}&\leq C\norm{\dot{\mathbf{u}}}_{L^2(\Omega_t)}+C\norm{\dot{\mathbf{u}}}_{L^2(\Omega_t)}^{\frac{2}{p}}\norm{\nabla\dot{\mathbf{u}}}_{L^2(\Omega_t)}^{1-\frac{2}{p}}\\
&\leq C(\norm{\sqrt{\rho}\dot{\mathbf{u}}}_{L^2(\Omega_t)}+\norm{\nabla\dot{\mathbf{u}}}_{L^2(\Omega_t)})\\
&\leq C(1+\norm{\div \dot{\mathbf{u}}}_{L^2(\Omega_t)}).
\end{split}
\end{align*}
Next, by standard $L^p$-estimate for elliptic system and Lemma \ref{lem 2.3}, we get
\begin{align*}
\begin{split}
\norm{\nabla\mathbf{u}}_{L^\infty(\Omega_t)}&\leq\norm{\nabla\mathbf{(\zeta \mathbf{u})}}_{L^\infty(\Omega_t)}+\norm{(1-\zeta)\nabla\mathbf{u}+\nabla\zeta\mathbf{u}}_{L^\infty(\Omega_t)}\\
&\leq C(1+\norm{\div \dot{\mathbf{u}}}_{L^2(\Omega_t)})\\
&\quad+C\norm{\div(\zeta\mathbf{u})}_{L^\infty(\Omega_t)}\log(e+\norm{\nabla^2(\zeta\mathbf{u})}_{L^p(\Omega_t)})+C\norm{\nabla(\zeta\mathbf{u})}_{L^2(\Omega_t)}\\
&\leq C(1+\norm{\div \dot{\mathbf{u}}}_{L^2(\Omega_t)})\\
&\quad+C(1+\norm{F}_{L^\infty(\Omega_t)})\log(e+\norm{\nabla\div(\zeta\mathbf{u})}_{L^p(\Omega_t)})+C\norm{\div \mathbf{u}}_{L^2(\Omega_t)}\\
&\leq C(1+\norm{\div \dot{\mathbf{u}}}_{L^2(\Omega_t)})\\
&\quad+C\left(1+\norm{F}_{L^2(\Omega_t)}^{\frac{p-2}{2p-2}}\norm{\nabla F}_{L^p(\Omega_t)}^{\frac{p}{2p-2}}\right)\log(e+\norm{\nabla\div(\zeta\mathbf{u})}_{L^p(\Omega_t)})+C\norm{\div \mathbf{u}}_{L^2(\Omega_t)}\\
&\leq C(1+\norm{\div \dot{\mathbf{u}}}_{L^2(\Omega_t)})+C(1+\norm{\div \dot{\mathbf{u}}}_{L^2(\Omega_t)}^{\frac{p}{2p-2}})\log(e+\norm{\nabla F}_{L^p(\Omega_t)}+\norm{\nabla\rho}_{L^p(\Omega_t)})\\
&\leq C(1+\norm{\div \dot{\mathbf{u}}}_{L^2(\Omega_t)})\log(e+\norm{\nabla\rho}_{L^p(\Omega_t)}),
\end{split}
\end{align*}
which we used 
\begin{align*}
\begin{split}
\norm{\nabla\zeta \mathbf{u}}_{L^\infty(\Omega_t)}&\leq C\norm{\div \mathbf{u}}_{L^2(\Omega_t)}\leq C\norm{F}_{L^2(\Omega_t)}+C\leq C
\end{split}
\end{align*}
and
\begin{align*}
\begin{split}
\norm{(1-\zeta)\nabla\mathbf{u}}_{L^\infty(\Omega_t)}&=\norm{(1-\zeta)\sqrt{(\partial_ru)^2+\left(\frac{u}{r}\right)^2}}_{L^\infty(\Omega_t)}\leq C+C\norm{(1-\zeta)(\partial_ru)}_{L^\infty(\Omega_t)}\\
&\leq C+C\norm{(1-\zeta)\sqrt{\left(\partial_ru+\frac{u}{r}\right)^2}}_{L^\infty(\Omega_t)}\leq C+C\norm{(1-\zeta)F}_{L^\infty(\Omega_t)}\\
&\leq C+C\norm{F}_{L^2(\Omega_t)}^{\frac{p-2}{2p-2}}\norm{\nabla F}_{L^p(\Omega_t)}^{\frac{p}{2p-2}}\leq C+C\norm{\div \dot{\mathbf{u}}}_{L^2(\Omega_t)}.
\end{split}
\end{align*}
Thus, by Gronwall's inequality, we get
\begin{align*}
\sup_{0\leq t\leq T}\norm{\nabla\rho}_{L^p(\Omega_t)}\leq C.
\end{align*}
Finally, we have
\begin{align*}
\begin{split}
\norm{\nabla^2\mathbf{u}}_{L^2(\Omega_t)}&\leq\norm{\nabla^2(\zeta \mathbf{u})}_{L^2(\Omega_t)}+\norm{(1-\zeta)\nabla^2\mathbf{u}}_{L^2(\Omega_t)}+\norm{\nabla\zeta\nabla\mathbf{u}}_{L^2(\Omega_t)}+C\\
&\leq C\norm{\nabla\div(\zeta\mathbf{u})}_{L^2(\Omega_t)}+C\left(1+\norm{(1-\zeta)\nabla F}_{L^2(\Omega_t)}\right.\\
& \left.\quad +\norm{(1-\zeta)\nabla\rho}_{L^2(\Omega_t)}+\norm{F}_{L^2(\Omega_t)}\right)\\
&\leq C+C\norm{\zeta\nabla\div\mathbf{u}}_{L^2(\Omega_t)}+C\norm{\nabla(\mathbf{u}\cdot \nabla\zeta)}_{L^2(\Omega_t)}\\
&\leq C,
\end{split}
\end{align*}
which we used the fact
\begin{align*}
\begin{split}
&\quad\norm{(1-\zeta)\nabla^2\mathbf{u}}_{L^2(\Omega_t)}\\
&\leq C\norm{(1-\zeta)\partial_{rr}u}_{L^2(\Omega_t)}+C\norm{(1-\zeta)\partial_{r}u}_{L^2(\Omega_t)}+C\norm{(1-\zeta)u}_{L^2(\Omega_t)}\\
&\leq C\norm{(1-\zeta)\nabla(\div\mathbf{u})}_{L^2(\Omega_t)}+C\norm{(1-\zeta)\nabla\rho}_{L^2(\Omega_t)}+C\norm{(1-\zeta)\div\mathbf{u}}_{L^2(\Omega_t)}+C\\
&\leq C\left(1+\norm{(1-\zeta)\nabla F}_{L^2(\Omega_t)} +\norm{(1-\zeta)\nabla\rho}_{L^2(\Omega_t)}+\norm{F}_{L^2(\Omega_t)}\right).
\end{split}
\end{align*}
\end{proof}

The remained procedure is standard: we can construct global approximate solutions to the stress-free boundary problem, and then show the convergence to a solution of the original system  based on the above a priori estimates. 

Here we give the key approximated scheme as follows: The approximation of Lagrangian system is constructed as
\begin{equation*}
	\begin{dcases}
		\rho_\tau^{(k+1)} +\br{\rho^{(k+1)} }^2 \br{r^{(k)} u^{(k)} }_x=0,\\
		\frac{1}{r^{(k)}} u_\tau^{(k+1) } +\br{(\rho^{(k)} )^\gamma -(2\mu+\rho^{(k)})\rho^{(k)}\br{r^{(k)}u^{(k+1)} }_x }_x=0,
	\end{dcases}
\end{equation*}
where $\frac{d}{d\tau} r^{(k)}(x,\tau)=u^{(k)}(x,\tau) $. Then we impose initial data
\begin{equation*}
		\br{\rho^{(k+1)}, u^{(k+1)} }(x,0)=(\rho_0^{\epsilon}, u_0^\epsilon)=\br{\frac{\rho_0+\epsilon}{\int_0^{a_0}(\rho_0 +\epsilon ) rdr}, \frac{\rho_0u_0}{\rho_0+\epsilon } \int_0^{a_0}(\rho_0 +\epsilon ) rdr},
\end{equation*}
and the boundary condition
\begin{equation*}
	u^{(k+1)}(0,\tau)=0, \br{(\rho^{(k)} -(2\mu+\rho^{(k)})\rho^{(k)}\br{r^{(k)}u^{(k+1)} }_x }(1,\tau)=0.
\end{equation*}
Using the contraction mapping theorem and energy estimates, we can prove the global existence of the solution of above approximated system. Since a priori estimates are uniform with respect to $\epsilon$ and $k$, we can show the convergence of the approximated solutions. It is direct to check the limit is exactly the unique global classic solution of the original system. The proof of Theorem \ref{global existence} is finished.

\section*{Ackonwledgments}
X.-D. Huang is partially supported by CAS Project for Young Scientists in Basic Research,
Grant No. YSBR-031 and NNSFC Grant Nos. 11688101.

\end{document}